\documentclass[]{elsarticle}
\usepackage{amssymb, amsmath, mathrsfs, amsthm, natbib,
hyperref}

\theoremstyle{definition}
\newtheorem{define}{Definition}[section]
\newtheorem{thm}[define]{Theorem}
\newtheorem{lemma}[define]{Lemma}
\newtheorem{remark}[define]{Remark}

\newtheorem{cor}[define]{Corollary}

\numberwithin{equation}{section}

\begin{document}

\begin{frontmatter}

\title{$k$-isomorphism classes of local field
extensions}
\author{Duc Van Huynh \fnref{fn1}}
\ead{duc.huynh@armstrong.edu}
\author{Kevin Keating}
\ead{keating@ufl.edu}
\address{Department of Mathematics, University of
Florida, Gainesville, FL 32611-8105, USA}
\fntext[fn1]{Current address: Department of Mathematics, 
Armstrong State University, Savannah, GA 31419}
\date{\today}

\begin{abstract}
Let $K$ be a local field of characteristic $p$ with
perfect residue field $k$.  In this paper we find a set
of representatives for the $k$-isomorphism classes of
totally ramified separable extensions $L/K$ of degree
$p$.  This extends work of Klopsch, who found
representatives for the $k$-isomorphism classes of
totally ramified Galois extensions $L/K$ of degree $p$.
\end{abstract}

\end{frontmatter}

\section{Introduction and results}

Let $K$ be a local field with perfect residue field $k$
and let $K_{s}$ be a separable closure of $K$.  The
problem of enumerating finite subextensions $L/K$ of
$K_{s}/K$ has a long history (see for instance
\cite{krasner}).  Alternatively, one might wish to
enumerate isomorphism classes of extensions.  Say that
the finite extensions $L_1/K$ and $L_2/K$ are {\em
$K$-isomorphic} if there is a field isomorphism $\sigma
: L_1 \rightarrow L_2$ which induces the identity map on
$K$.  In this case the extensions $L_1/K$ and $L_2/K$
share the same field-theoretic and arithmetic data; for
instance their degrees, automorphism groups, and
ramification data must be the same.  In the case where
$K$ is a finite extension of the $p$-adic field
$\mathbb{Q}_p$, Monge \cite{monge} computed the number
of $K$-isomorphism classes of extensions $L/K$ of degree
$n$, for arbitrary $n\ge1$.

One says that the finite extensions $L_1/K$ and $L_2/K$
are {\em $k$-isomorphic} if there is a field isomorphism
$\sigma : L_1 \rightarrow L_2$ such that $\sigma(K) = K$
and $\sigma$ induces the identity map on $k$.  Such an
isomorphism is automatically continuous (see
Lemma~\ref{contin}).  If the
extensions $L_1/K$ and $L_2/K$ are $k$-isomorphic then
they have the same field-theoretic and arithmetic
properties.  Let $\text{Aut}_k(K)$ denote the group of
field automorphisms of $K$ which induce the identity map
on $k$.  Then $\text{Aut}_k(K)$ is finite if char$(K) =
0$, infinite if char$(K) = p$.  Since every
$k$-isomorphism $\sigma$ from $L_1/K$ to $L_2/K$ induces
an element of $\text{Aut}_k(K)$, this suggests that
$k$-isomorphisms should be more plentiful when
char$(K)=p$.  In this paper we consider the problem of
classifying $k$-isomorphism classes of finite totally
ramified extensions of a local field $K$ of
characteristic $p$.

As one might expect, the tame case is straightforward:
It is easily seen that if $n \in \mathbb{N}$ is
relatively prime to $p$ then there is a unique
$k$-isomorphism class of totally ramified extensions
$L/K$ of degree $n$.  We will focus on ramified
extensions of degree $p$, which are the simplest
non-tame extensions.  Since any two $k$-isomorphic
extensions have the same ramification data, it makes
sense to classify $k$-isomorphism classes of degree-$p$
extensions with fixed ramification break $b > 0$.

Let $\mathcal{E}_{b}$ denote the set of all totally
ramified subextensions of $K_{s}/K$ of degree $p$ with
ramification break $b$, and let $\mathcal{S}_{b}$ denote
the set of $k$-isomorphism classes of elements of
$\mathcal{E}_{b}$.  Let $\mathcal{S}_{b}^g$ denote the
set of $k$-isomorphism classes of Galois extensions in
$\mathcal{E}_{b}$, and let $\mathcal{S}_{b}^{ng}$ denote
the set of $k$-isomorphism classes of non-Galois
extensions in $\mathcal{E}_{b}$.  As we will see in
Section~\ref{amanowork}, if $b$ is the ramification
break of an extension of degree $p$ then $(p-1)b \in
\mathbb{N} \smallsetminus p \mathbb{N}$.  Hence
$\mathcal{S}_{b}$ is empty if $b \not \in \frac{1}{p-1}
\cdot (\mathbb{N} \smallsetminus p \mathbb{N})$.

\begin{thm} \label{theorem}
Let $b \in \frac{1}{p-1} \cdot (\mathbb{N}
\smallsetminus p \mathbb{N})$ and write $b =
\frac{(m-1)p + \lambda}{p-1}$ with $1 \le \lambda \le
p-1$.  Let $R = \{\omega_i : i \in I\}$ be a set of
coset representatives for $k^{\times} /
(k^{\times})^{(p-1) b}$.  For each $\omega_i \in R$ let
$\pi_i \in K_{s}$ be a root of the
polynomial $X^p - \omega_i \pi_K^m X^{\lambda} - \pi_K$.
Then the map which carries $\omega_i$ onto the
$k$-isomorphism class of $K(\pi_i)/K$ gives a
bijection from $R$ to $\mathcal{S}_{b}$.  Furthermore,
$K(\pi_i)/K$ is Galois if and only if $b \in \mathbb{N}
\smallsetminus p \mathbb{N}$ and $\lambda \omega_i \in
(k^{\times})^{p-1}$.
\end{thm}

\begin{cor} \label{number}
Let $b \in \frac{1}{p-1} \cdot (\mathbb{N}
\smallsetminus p \mathbb{N})$ and assume that $|k| = q <
\infty$.  Then
\[|\mathcal{S}_{b}| = \gcd(q - 1 , ( p - 1) b).\]
Furthermore, if $b \in \mathbb{N} \smallsetminus p
\mathbb{N}$ then
\begin{align*}
\left|\mathcal{S}_{b}^g \right| & =
\gcd \left(\frac{q-1}{p-1}, b \right) \\
|\mathcal{S}_{b}^{ng}| & =
(p-2) \cdot \gcd \left(\frac{q-1}{p-1}, b \right).
\end{align*}
\end{cor}

\begin{proof}
This follows from Theorem~\ref{theorem} and the formulas
\begin{alignat*}{2}
|k^{\times} / (k^{\times})^{(p-1) b}| & = 
\gcd( q - 1 , ( p - 1 ) b), \\
|(k^{\times})^{p-1} / (k^{\times})^{(p-1) b}| & = 
\gcd\left(\frac{q - 1}{p-1} , b \right) && \text{for }b \in
\mathbb{N} \smallsetminus p \mathbb{N}.
\end{alignat*}
\end{proof}

The proof of Theorem~\ref{theorem} relies heavily on
the work of Amano, who showed in \cite{amano} that every
degree-$p$ extension of a local field of characteristic
0 is generated by a root of an Eisenstein
polynomial with a special form, which we call an
\textit{Amano polynomial} (see Definition~\ref{poly}).
In Section~\ref{amanowork} we show how Amano's results
can be adapted to the characteristic-$p$ setting.  In
Section~\ref{action} we prove Theorem~\ref{theorem} by
computing the orbits of the action of $\text{Aut}_k(K)$
on the set of Amano polynomials over $K$.

\section{Amano polynomials in characteristic
\texorpdfstring{$p$}{p}}
\label{amanowork}

Let $F$ be a finite extension of the $p$-adic field
$\mathbb{Q}_p$ and let $E/F$ be a totally ramified
extension of degree $p$.  In \cite{amano}, Amano
constructs an Eisenstein polynomial $g(X)$ over $F$ with
at most 3 terms such that $E$ is generated over $F$ by a
root of $g(X)$.  In this section we reproduce a part of
Amano's construction in characteristic $p$.  We
associate a family of 3-term Eisenstein polynomials to
each ramified separable extension of $L/K$ of degree
$p$, but we don't choose representatives for these
families.  Many of the proofs from \cite{amano} remain
valid in this new setting.

Let $K$ be a local field of characteristic $p$ with
perfect residue field $k$.  Let $K_s$ be a separable
closure of $K$ and let $\nu_K$ be the valuation of $K_s$
normalized so that $\nu_K(K^{\times})={\mathbb Z}$.
Fix a prime element $\pi_K$
for $K$; since $k$ is perfect we may identify $K$ with
$k((\pi_K))$.  Let $U_K$ denote the group of units of
$K$, and let $U_{1,K}$ denote the subgroup of 1-units.
If $u \in U_{1,K}$ and $\alpha \in
\mathbb{Z}_{p}$ is a $p$-adic integer then $u^{\alpha}$
is defined as a limit of positive integer powers of $u$.
This applies in particular when $\alpha$ is a rational
number whose denominator is not divisible by $p$.

Let $L/K$ be a finite totally ramified subextension of
$K_s/K$ and let $\nu_L$ be the valuation of $K_s$
normalized so that $\nu_L(L^{\times})={\mathbb Z}$.  Let
$\pi_L$ be a prime element for $L$ and let
$\sigma:L\rightarrow K_s$ be a $K$-embedding of $L$ into
$K_s$, such that $\sigma\not=\text{id}_L$.  We define
the ramification number of $\sigma$ to be
$\nu_L(\sigma(\pi_L)-\pi_L)-1$.  It is easily seen that
this definition does not depend on the choice of
$\pi_L$.  We say that $b$ is a (lower) ramification
break of the extension $L/K$ if $b$ is the ramification
number of some nonidentity $K$-embedding of $L$ into
$K_s$.

Suppose $L/K$ is a separable totally ramified extension
of degree $p$.  Then Lemma 1 of \cite{amano} shows that
$L/K$ has a unique ramification break.  Every prime
element $\pi_L$ of $L$ is a root of an Eisenstein
polynomial
\[
f(X) = X^p - \sum_{i=0}^{p-1} c_i X^i
\]
over $K$, with $\nu_K(c_0) = 1$ and $\nu_K(c_i) \geq 1$
for $1 \leq i \leq p-1$.  Let $\pi_L' \not = \pi_L$ be a
conjugate of $\pi_L$ in $K_{s}$.  Then the ramification
break of $L/K$ is given by
\[
b = \nu_L\left(\frac{\pi_L'}{\pi_L} - 1\right).
\]

Since $L/K$ is separable, we have $c_i \not= 0$ for some
$i$ with $1 \le i \le p-1$.  Therefore
\[
m = \min\{\nu_K(c_1), \ldots, \nu_K(c_{p-1})\}
\]
is finite.  Let $\lambda$ be minimum such that
$\nu_K(c_{\lambda}) = m$ and let $\omega \in k^{\times}$
satisfy $c_{\lambda} \equiv \omega \pi_K^m
\pmod{\pi_K^{m+1}}$.  We say that the Eisenstein
polynomial $f(X)$ is of type $\langle \lambda,m,\omega
\rangle$.  Note that while $\omega$ depends on the
choice of $\pi_K$, the positive integers $m$ and
$\lambda$ do not.  If $f(X)$ is of type $\langle
\lambda,m,\omega \rangle$ then by Lemma~1 of
\cite{amano} the ramification break $b$ of $L/K$ is
given by
\begin{equation} \label{break}
b = \frac{(m-1)p + \lambda}{p-1}.
\end{equation}
Conversely, given $b \in \frac{1}{p-1} \cdot
(\mathbb{N} \smallsetminus p \mathbb{N})$, equation
(\ref{break}) uniquely determines $m$ and $\lambda$, and
we can easily construct Eisenstein polynomials of type
$\langle \lambda,m,\omega \rangle$ for every $\omega \in
k^{\times}$.

For Eisenstein polynomials $f(X), g(X) \in K[X]$, write
$f(X) \sim g(X)$ if there is a $K$-isomorphism
\[
K[X]/(f(X)) \cong K[X]/(g(X)).
\]
Then $\sim$ is an equivalence relation on Eisenstein
polynomials over $K$.

\begin{thm} \label{sametype}
Suppose $f(X),g(X) \in K[X]$ are Eisenstein polynomials
of degree $p$ such that $f(X) \sim g(X)$.  Then $f(X)$
and $g(X)$ are of the same type.
\end{thm}

\begin{proof}
The proof of Theorem~1 of \cite{amano} applies here,
except that in characteristic $p$ we don't have to
consider polynomials of type $\langle0\rangle$.
\end{proof}

Henceforth we say that an extension $L/K$ has type
$\langle \lambda,m,\omega \rangle$ if $L/K$ is
$K$-isomorphic to $K[X]/(f(X))$ for some Eisenstein
polynomial $f(X)$ of type $\langle \lambda,m,\omega
\rangle$.

\begin{thm} \label{galois}
Let $L/K$ be an extension of type $\langle \lambda,m,\omega
\rangle$. Then $L/K$ is Galois if and only $b =
\frac{(m-1)p + \lambda}{p-1}$ is
an integer and $\lambda \omega \in (k^{\times})^{p-1}$.
\end{thm}

\begin{proof}
The proof of Theorem~3(ii) of \cite{amano} applies
without change.
\end{proof}

\begin{thm} \label{Apoly}
Suppose $L/K$ is an extension of type $\langle
\lambda,m,\omega \rangle$.
Then there exists a prime element $\pi_L \in L$ which is
a root of a polynomial
\[
A_{\omega,u}^{b}(X) = X^p - \omega \pi_K^m
X^{\lambda} - u \pi_K
\]
for some $u \in U_{1,K}$.
\end{thm}

\begin{proof}
The proof of Theorem~4 of \cite{amano} applies here,
except that we don't have to consider extensions of type
$\langle0\rangle$.  Briefly, one defines a function
$\phi:L\rightarrow K$ by
\[
\phi(\alpha)=\alpha^p-\omega \pi_K^m \alpha^n -
{\rm N}_{L/K}(\alpha),
\]
where ${\rm N}_{L/K}$ is the norm from $L$ to $K$.
Using an iterative procedure one gets a prime element
$\pi$ in $L$ such that $\nu_L(\phi(\pi)) >
p(\lambda+1)$ and ${\rm N}_{L/K}(\pi) = u\pi_K$ for some
$u \in U_{1,K}$.  Let $\pi^{(1)},\dots,\pi^{(p)}\in
K_s$ be the roots of $A_{\omega, u}^b(X)$.  Then
\begin{equation}
\phi(\pi) = A_{\omega, u}^b(\pi) = \prod_{i=1}^p (\pi - \pi^{(i)}),
\end{equation}
so we have
\begin{equation}
\sum_{i=1}^p \nu_L(\pi - \pi^{(i)}) = \nu_L(\phi(\pi)) >
p(\lambda+1).
\end{equation}
Hence $\nu_L(\pi - \pi^{(j)}) > \lambda + 1$ for some
$j$, so we get $L \subset K(\pi^{(j)})$ by Krasner's
Lemma.  Since $[K(\pi^{(j)}):K]=[L:K]=p$, it follows
that $L = K(\pi^{(j)})$.  Therefore $\pi_L = \pi^{(j)}$
satisfies the conditions of the theorem.
\end{proof}

\begin{define} \label{poly}
We say that $A_{\omega,u}^{b}(X)$ is an {\em Amano
polynomial} over $K$ with ramification break $b$.
\end{define}

Let $b = \frac{(m-1) p + \lambda}{p - 1}$ with $1 \le
\lambda \le p-1$.  We denote the set of Amano
polynomials over $K$ with ramification break $b$ by
\[
\mathscr{P}_{b} = \{X^p - \omega \pi_K^m X^{\lambda}
- u \pi_K : \omega \in k^{\times}, \; u \in U_{1,K} \}.
\]
Let $\mathscr{P}_{b} / {\sim}$ denote the set of
equivalence classes of $\mathscr{P}_{b}$ with respect to
$\sim$.  For $f(X) \in \mathscr{P}_{b}$, we denote the
equivalence class of $f(X)$ by $[f(X)]$.  It follows
from Theorem~\ref{Apoly} that these
equivalence classes are in one-to-one correspondence
with the elements of $\mathcal{E}_b$.

\section{The action of
\texorpdfstring{Aut$_k(K)$}{Aut(K)} on extensions}
\label{action}

In this section we show how Aut$_k(K)$ acts on the set
of equivalence classes of Amano polynomials with
ramification break $b$.  We determine the orbits of this
action, and give a representative for each orbit.  This
allows us to construct representatives for the elements
of $\mathcal{S}_b$, and leads to the proof of
Theorem~\ref{theorem}.

The following lemma is certainly well-known (see, for
instance, the answers to \cite{overflow}), but we could
find no reference for it.

\begin{lemma} \label{contin}
Let $L_1$ and $L_2$ be local fields.  Assume that $L_1$
and $L_2$ have the same residue field $k$, and that $k$
is a perfect field of characteristic $p$.  Let
$\sigma:L_1\rightarrow L_2$ be a field isomorphism.
Then $\nu_{L_2}\circ\sigma=\nu_{L_1}$.
\end{lemma}

\begin{proof} The group $U_{1,L_1}$ is $n$-divisible
for all $n$ prime to $p$, so we have
$\sigma(U_{1,L_1})\subset U_{L_2}$.  For $i=1,2$ the
group $T_i$ of
nonzero Teichm\"uller representatives of $L_i$ is equal
to $\bigcap_{i=1}^{\infty}(L_i^{\times})^{p^i}$, so we
have $\sigma(T_1)=T_2$.  Since $U_{L_i}=T_i\cdot
U_{L_i,1}$ this implies $\sigma(U_{L_1})\subset
U_{L_2}$.  The same reasoning shows that
$\sigma^{-1}(U_{L_2}) \subset U_{L_1}$, so we get
$\sigma(U_{L_1})=U_{L_2}$.  It follows that
$\nu_{L_2}\circ\sigma$, like $\nu_{L_1}$, induces an
isomorphism of $L_1^{\times}/U_{L_1}$ onto $\mathbb Z$.
Let $\pi_{L_1}$ be a prime element of $L_1$.  Then
$1+\pi_{L_1}\in U_{L_1,1}$, so
$\nu_{L_2}(\sigma(1+\pi_{L_1}))=0$.  Hence
$\nu_{L_2}(\sigma(\pi_{L_1}))\ge0$.  Since
$\nu_{L_2}(\sigma(\pi_{L_1}))$ generates $\mathbb Z$,
it follows that $\nu_{L_2}(\sigma(\pi_{L_1}))=1$.
We conclude that $\nu_{L_2}\circ\sigma=\nu_{L_1}$.
\end{proof}

For $f(X) \in K[X]$ and $\varphi \in \text{Aut}_k(K)$ we
let $f^{\varphi}(X)$ denote the polynomial obtained by
applying $\varphi$ to the coefficients of $f(X)$.  The
following lemma is a straightforward ``transport of
structure'' result:

\begin{lemma} \label{well}
Let $f(X)$ and $g(X)$ be Eisenstein polynomials with
coefficients in $K$ such that $f(X) \sim g(X)$, and let
$\varphi \in \text{Aut}_k(K)$.  Then $f^{\varphi}(X)
\sim g^{\varphi}(X)$.
\end{lemma}

Let $\mathscr{A} = \text{Aut}_k(K)$ denote the group of
$k$-automorphisms of $K$.  Since all $k$-automorphisms
of $K = k((\pi_K))$ are continuous by
Lemma~\ref{contin}, every $\varphi \in \mathscr{A}$ is
determined by the value of
$\varphi(\pi_K)$.  Furthermore, $\mathscr{A}$ acts
transitively on the set of prime elements of $K$.
It follows that the group consisting of the power series
\[\left\{\sum_{i=1}^{\infty} \, a_i t^i : a_i \in k, \;
a_1 \not= 0 \right\}\]
with the operation of substitution is isomorphic to the
opposite group $\mathscr{A}^{op}$ of $\mathscr{A}$.  For
every $\varphi \in \mathscr{A}$ there are $l_{\varphi}
\in k^{\times}$
and $v_{\varphi} \in U_{1,K}$ such that $\varphi(\pi_K)
= l_{\varphi} \cdot v_{\varphi} \cdot \pi_K$.  Let
\[
\mathscr{N} = \{\sigma \in \mathscr{A}:
\sigma(\pi_K) \in U_{1,K} \cdot \pi_K \}
\]
be the group of wild automorphisms of $K$.  Then
$\mathscr{N}^{op}$ is isomorphic to the Nottingham Group
over $k$ (see \cite{klopsch}).  Furthermore, $\mathscr{N}$
is normal in $\mathscr{A}$, and $\mathscr{A} / \mathscr{N}
\cong k^{\times}$. 

Let $\varphi \in \mathscr{A}$ and let
$A_{\omega,u}^{b}(X) \in \mathscr{P}_{b}$.  Then
by Theorem~\ref{Apoly} there exist $\omega' \in
k^{\times}$ and $u' \in U_{1,K}$ such that
\begin{align*}
K[X]/((A_{\omega,u}^{b})^{\varphi}(X)) & = K[X]/(X^p -
\varphi(\omega \pi_K^m)X^{\lambda} - \varphi(\pi_K u)) \\
& \cong K[X]/(A_{\omega',u'}^{b}(X)).
\end{align*}
It follows from Lemma~\ref{well} that
\begin{equation} \label{phiA}
\varphi \cdot [A_{\omega,u}^{b}(X)] =
[A_{\omega',u'}^{b}(X)]
\end{equation}
gives a well-defined action of $\mathscr{A}$ on
$\mathscr{P}_{b} / {\sim}$.  The following theorem
computes explicit values for $\omega'$ and $u'$ in
(\ref{phiA}).  Note that since $k$ is perfect,
$l_{\varphi}$ has a unique $p$th root
$l_{\varphi}^{\frac{1}{p}}$ in $k$.

\begin{thm} \label{specific}
Let $\varphi \in \mathscr{A}$ and
$A_{\omega,u}^{b}(X) \in \mathscr{P}_{b}$. 
Then $\varphi \cdot [A_{\omega,u}^{b}(X)] = [A_{\omega'
, u' }^{b}(X)]$, with
$\omega' = \omega \cdot l_{\varphi}^{\frac{(p-1) b}{p}}$,
$u' = \varphi(u) \cdot v_{\varphi}^h$, and $h =
\frac{p-\lambda-pm}{p-\lambda}$.
\end{thm}

\begin{proof}
By applying $\varphi$ to the coefficients of
$A_{\omega, u}^{b}(X)$ we get
\[
(A_{\omega, u}^{b})^{\varphi}(X) =
X^p - \omega l_{\varphi}^m v_{\varphi}^m \pi_K^m
X^{\lambda} - \varphi(u) l_{\varphi} v_{\varphi} \pi_K.
\]
Set $X = l_{\varphi}^{\frac{1}{p}}
v_{\varphi}^{\frac{m}{p-\lambda}} Z$.  Then
\begin{align*}
l_{\varphi}^{-1}v_{\varphi}^{\frac{-pm}{p-\lambda}}
(A_{\omega,u}^{b})^{\varphi}(X) & = Z^p - \omega
l_{\varphi}^{\frac{(p-1) b}{p}}
\pi_K^m Z^{\lambda} - \varphi(u) v_{\varphi}^h  \pi_K \\
& = Z^p - \omega' \pi_K^m Z^{\lambda} - u' \pi_K.
\end{align*}
Since $l_{\varphi}^{\frac{1}{p}}
v_{\varphi}^{\frac{m}{p-\lambda}} \in K$, it follows that
\[
K[X]/(A_{\omega,u}^{b})^{\varphi}(X)) \cong
K[X] / (A_{\omega',u'}^{b}(X)).
\]
\end{proof}

To determine the orbit of $[A_{\omega,u}^{b}(X)]$ under
the action of $\mathscr{A}$ we need the following
lemmas.  Let $\mathbb{Z}_{p}^{\times}$ denote the unit
group of the ring of $p$-adic integers.

\begin{lemma} \label{onto}
Let $u \in U_{1,K}$, and $h \in \mathbb{Z}_p^{\times}$.
Then
\[ U_{1,K} = \left \{ \sigma(u) \cdot \left(
\frac{\sigma(\pi_K)}{\pi_K} \right)^h : \; \sigma \in
\mathscr{N} \right \}. \]
\end{lemma}

\begin{proof}
Let $v = u^{\frac{1}{h}} \in U_{1,K}$.  Then $\pi_K' = v
\pi_K$ is a prime element of $K$.  We have
\begin{align*}
U_{1,K} 
&= \left \{ \frac{v \sigma(\pi_K')}{\pi_K'} : \sigma \in
\mathscr{N} \right \} \\
&= \left \{ \frac{\sigma(v \pi_K)}{\pi_K} : \sigma \in
\mathscr{N} \right \} \\
&= \left \{ \sigma(u)^{\frac{1}{h}} \cdot \frac{
\sigma(\pi_K)}{\pi_K} : \sigma \in \mathscr{N}
\right \}.
\end{align*}
Since $h \in \mathbb{Z}_p^{\times}$, we have $U_{1,K}^h
= U_{1,K}$.  Hence by raising to the power $h$ we
obtain
\[
U_{1,K} = \left \{ \sigma(u) \cdot \left(
\frac{\sigma(\pi_K)}{\pi_K} \right)^h : \; \sigma \in
\mathscr{N} \right \}.
\]
\end{proof}

\begin{lemma} \label{same}
Let $c \in k^{\times}$ and define $\tau_c \in \mathscr{A}$
by $\tau_c(\pi_K) = c\pi_K$.  Let $\mathscr{N}_c =
\mathscr{N} \tau_c$ be the right coset of $\mathscr{N}$
in $\mathscr{A}$ represented by $\tau_c$.  Then for
$u \in U_{1,K}$ and $h \in \mathbb{Z}_p^{\times}$ we
have
\[
U_{1,K} = \{ \varphi(u) \cdot v_{\varphi}^h : \varphi
\in \mathscr{N}_c \}.
\]
\end{lemma}

\begin{proof}
Let $u' = \tau_c(u) \in U_{1,K}$.  Then
\begin{align*}
\{ \varphi(u) \cdot v_{\varphi}^h  : \varphi \in
\mathscr{N}_c \} & = \{ \sigma \tau_c(u) \cdot
v_{\sigma}^h : \sigma \in \mathscr{N} \} \\
& = \{ \sigma (u') \cdot v_{\sigma}^h : \sigma \in
\mathscr{N} \} \\ 
& = U_{1,K},
\end{align*}
where the last equality follows from Lemma~\ref{onto}.
\end{proof}

\begin{thm} \label{orbit}
The orbit of $[A_{\omega,u}^{b}(X)]$ under $\mathscr{A}$
is
\[
\mathscr{A} \cdot [A_{\omega,u}^{b}(X)] =
\{ [A_{\omega \theta,v}^{b}(X)] : \theta \in
(k^{\times})^{(p-1)b}, \; v \in U_{1,K} \}.
\]
\end{thm}

\begin{proof}
Let $c \in k^{\times}$ and $\varphi \in \mathscr{N}_c$.
Then $l_{\varphi} = c$, so by Theorem~\ref{specific} we
have
\[\varphi \cdot [A_{\omega,u}^{b}(X)] =
[A_{\omega',u'}], \]
with $\omega' = \omega c^{\frac{(p-1)b}{p}}$, $u' =
\varphi(u) v_{\varphi}^{h}$, and $h = \frac{p - \lambda
- pm}{p - \lambda}$.  Hence by Lemma~\ref{same} we have
\[\mathscr{N}_c \cdot [A_{\omega,u}^{b}(X)] = \{
[A_{\omega',v}] : \omega' = \omega c^{\frac{(p-1)b}{p}},
\; v \in U_{1,K}\}. \]
Since $\mathscr{A}$ is the union of $\mathscr{N}_c$ over
all $c \in k^{\times}$, and $k$ is perfect, the theorem
follows.
\end{proof}

We now give the proof of Theorem~\ref{theorem}.  Let
$R = \{\omega_i : i \in I\}$ be a set of coset
representatives for $k^{\times} /
(k^{\times})^{(p-1)b}$.  For each $\omega_i \in R$ let
$\pi_i \in K_{s}$ be a root of the Amano polynomial
\[
A_{\omega_i,1}^{b}(X) = X^p - \omega_i \pi_K^m
X^{\lambda} - \pi_K.
\]
It follows from Theorem~\ref{orbit} that for every
equivalence class $\mathcal{C} \in \mathcal{S}_{b}$
there is $i \in I$ such that $K(\pi_i)/K \in
\mathcal{C}$.  On the other hand, if $K(\pi_i)/K$ is
$k$-isomorphic to $K(\pi_j)/K$ then by
Theorem~\ref{specific}, for some $\varphi \in
\mathscr{A}$ we have 
\[
[A_{\omega_j,1}^b(X)] = \varphi \cdot
[A_{\omega_i,1}^b(X)] =
[A_{\omega_i', v_{\varphi}^h}^{b}(X)].
\]
with $\omega_i' = \omega_i
l_{\varphi}^{\frac{(p-1)b}{p}}$.  It follows from
Theorem~\ref{sametype} that $A_{\omega_j,1}^b(X)$ and
$A_{\omega_i',
v_{\varphi}^h}^{b}(X)$ have the same type, so we have
$\omega_j = \omega_i l_{\varphi}^{\frac{(p-1)b}{p}}$.
Since $l_{\varphi}^{\frac{1}{p}} \in k^{\times}$, this
implies that $\omega_j \omega_i^{-1} \in
(k^{\times})^{(p-1)b}$.  Since $\omega_i$ and $\omega_j$
are coset representatives for $k^{\times} /
(k^{\times})^{(p-1)b}$, we get $\omega_i = \omega_j$.
This proves the first part of Theorem~\ref{theorem}.
The second part follows from Theorem~\ref{galois}.

\begin{remark} \label{klopsch}
In \cite{klopsch}, Klopsch uses a different method to
compute the cardinality of $\mathcal{S}_{b}^g$.
Let $L=k((\pi_L))$ be a local function field with
residue field $k$, and set $\mathscr{F} =
\text{Aut}_k(L)$.  Then there is a one-to-one
correspondence between cyclic subgroups $G \le
\mathscr{F}$ of order $p$ and subfields $M = L^{G}$ of
$L$ such that $L / M$ is a cyclic totally ramified
extension of degree $p$.  For $i=1,2$ let $G_i$ be a
cyclic subgroup of $\mathscr{F}$ of order
$p$ and set $K_i = L^{G_i}$.  Say the extensions
$L/K_1$ and $L/K_2$ are $k^*$-isomorphic if there exists
$\eta \in \mathscr{F} = \text{Aut}_k(L)$ such that
$\eta(K_1) = K_2$; this is equivalent to
$\eta^{-1} G_1 \eta = G_2$.

For $i=1,2$ let $\psi_i : K \rightarrow L$ be a
$k$-linear field embedding such that $\psi_i(K)=K_i$.
We can use $\psi_i$ to
identify $K$ with $K_i$, which makes $L$ an extension of
$K$.  We easily see that the extensions $\psi_1 : K
\hookrightarrow L$ and $\psi_2 : K \hookrightarrow L$ are
$k$-isomorphic if and only if $L/K_1$ and $L/K_2$ are
$k^*$-isomorphic.  Therefore classifying $k$-isomorphism
classes of degree-$p$ Galois extensions of $K$ is
equivalent to classifying conjugacy classes of subgroups
of order $p$ in $\mathscr{F}$.

For $i=1,2$ let $G_i = \langle \gamma_i \rangle$.  If
$G_1$ and $G_2$ have ramification break $b$ then
\begin{align*}
\gamma_1(\pi_L) & \equiv \pi_L + r_{b + 1}
\pi_L^{b + 1} \pmod{\pi_L^{b + 2}} \\
\gamma_2(\pi_L) & \equiv \pi_L + s_{b + 1}
\pi_L^{b + 1} \pmod{\pi_L^{b + 2}}
\end{align*}
for some $r_{b + 1} , s_{b + 1} \in k^{\times}$.
Hence for $1 \leq j \le p-1$, we have
\[
\gamma_1^j(\pi_L) \equiv \pi_L + j r_{b+1}
\pi_L^{b+1} \pmod{\pi_L^{b + 2}}.
\]
By Proposition 3.3 of \cite{klopsch}, $\gamma_1^j$ and
$\gamma_2$ are conjugate in $\mathscr{F}$ if and only if
$s_{b+1} = j r_{b+1} t^{b}$, for some
$t \in k^{\times}$.  Therefore the subgroups $G_1$
and $G_2$ are conjugate in $\mathscr{F}$ if and only if
$s_{b+1} \in r_{b+1} \cdot \mathbb{F}_p^{\times} \cdot
(k^{\times})^{b}$.  It follows that the number of
conjugacy classes of subgroups of order $p$ with
ramification break $b$ is
\[|k^{\times} / (\mathbb{F}_p^{\times} \cdot
(k^{\times})^{b})| = |(k^{\times})^{p-1} /
(k^{\times})^{(p-1)b}|.\]
In particular, if $|k| = q < \infty$ then there are
$\displaystyle \gcd \left(\frac{q-1}{p-1},b\right)$ such
conjugacy classes, in agreement with
Corollary~\ref{number}.
\end{remark}

\end{document}